\documentclass[a4paper,12pt]{amsart}

\author{Giles Gardam}
\title{Detecting laws in power subgroups}

\usepackage[T1]{fontenc}
\usepackage[utf8]{inputenc}

\usepackage{amsfonts}
\usepackage{amsmath}
\usepackage{amssymb}
\usepackage{amsthm}
\usepackage{array}
\usepackage{fullpage}
\usepackage{graphicx}
\usepackage{mathpazo}
\usepackage{microtype}
\usepackage[onehalfspacing]{setspace}
\usepackage{textcomp}
\usepackage{xcolor}
\usepackage{xspace}

\usepackage{hyperref}

\theoremstyle{plain}
\newtheorem{thm}{Theorem}
\numberwithin{thm}{section}
\newtheorem{lem}[thm]{Lemma}
\newtheorem{prop}[thm]{Proposition}
\newtheorem{cor}[thm]{Corollary}
\newtheorem{keythm}{Theorem}

\newtheorem{keycor}{Corollary}

\newtheorem*{corNilpotencyClass}{Corollary \ref{cor:nilpotency_is_detectable}}
\newtheorem*{corEngel}{Corollary \ref{cor:Engel}}
\newtheorem*{thmLocallyNilpotentVariety}{Theorem \ref{thm:any_nilpotent_variety}}
\newtheorem*{thmMetabelian}{Theorem \ref{thm:metabelian_counterexample}}
\newtheorem*{thmAbelianPresentation}{Theorem \ref{thm:its_abelian}}

\theoremstyle{definition}
\newtheorem{defn}[thm]{Definition}

\newtheorem{eg}[thm]{Example}
\newtheorem{qn}{Question}

\theoremstyle{remark}
\newtheorem{rmk}[thm]{Remark}

\newcommand{\cA}{\mathcal{A}}
\newcommand{\cB}{\mathcal{B}}
\newcommand{\cC}{\mathcal{C}}
\newcommand{\cD}{\mathcal{D}}
\newcommand{\cL}{\mathcal{L}}
\newcommand{\cM}{\mathcal{M}}
\newcommand{\cN}{\mathcal{N}}
\newcommand{\cP}{\mathcal{P}}
\newcommand{\cQ}{\mathcal{Q}}
\newcommand{\cU}{\mathcal{U}}
\newcommand{\cV}{\mathcal{V}}
\newcommand{\cW}{\mathcal{W}}

\newcommand{\Z}{\mathbb{Z}}

\newcommand{\abs}[1]{\lvert #1 \rvert}
\newcommand{\epi}{\twoheadrightarrow}
\newcommand{\gen}[1]{\langle #1 \rangle}
\newcommand{\gp}[2]{\langle \, #1 \, | \, #2 \, \rangle}
\newcommand{\mat}[4]{\begin{pmatrix} #1 & #2 \\ #3 & #4 \end{pmatrix}}
\newcommand{\normal}{\vartriangleleft}
\newcommand{\ps}[2]{{#1^{#2}}}
\newcommand{\set}[2]{\left\{ #1 \,\middle|\, #2 \right\}}
\newcommand{\SL}{\operatorname{SL}}
\newcommand{\Zmod}[1]{{\mathbb{Z}_{#1}}}

\DeclareMathOperator{\nprod}{\odot}
\DeclareMathOperator{\Forall}{\forall}

\makeatletter
\DeclareFontFamily{OMX}{MnSymbolE}{}
\DeclareSymbolFont{MnLargeSymbols}{OMX}{MnSymbolE}{m}{n}
\SetSymbolFont{MnLargeSymbols}{bold}{OMX}{MnSymbolE}{b}{n}
\DeclareFontShape{OMX}{MnSymbolE}{m}{n}{
    <-6>  MnSymbolE5
   <6-7>  MnSymbolE6
   <7-8>  MnSymbolE7
   <8-9>  MnSymbolE8
   <9-10> MnSymbolE9
  <10-12> MnSymbolE10
  <12->   MnSymbolE12
}{}
\DeclareFontShape{OMX}{MnSymbolE}{b}{n}{
    <-6>  MnSymbolE-Bold5
   <6-7>  MnSymbolE-Bold6
   <7-8>  MnSymbolE-Bold7
   <8-9>  MnSymbolE-Bold8
   <9-10> MnSymbolE-Bold9
  <10-12> MnSymbolE-Bold10
  <12->   MnSymbolE-Bold12
}{}

\let\llangle\@undefined
\let\rrangle\@undefined
\DeclareMathDelimiter{\llangle}{\mathopen}%
                     {MnLargeSymbols}{'164}{MnLargeSymbols}{'164}
\DeclareMathDelimiter{\rrangle}{\mathclose}%
                     {MnLargeSymbols}{'171}{MnLargeSymbols}{'171}
\makeatother
\newcommand{\ncls}[1]{\llangle #1 \rrangle}

\mathchardef\ordinarycolon\mathcode`\:
\mathcode`\:=\string"8000
\begingroup \catcode`\:=\active
  \gdef:{\mathrel{\mathop\ordinarycolon}}
\endgroup

\thanks{This work was partially supported by the Clarendon Fund, Balliol College Marvin Bower Scholarship, and the James Fairfax Oxford Australia Scholarship}
\address{Department of Mathematics \\ Technion \\ Haifa \\ Israel}
\email{gilesgar@technion.ac.il}


\begin{document}

\begin{abstract}
A group law is said to be detectable in power subgroups if, for all coprime $m$ and $n$, a group $G$ satisfies the law if and only if the power subgroups $\ps{G}{m}$ and $\ps{G}{n}$ both satisfy the law.
We prove that for all positive integers $c$, nilpotency of class at most $c$ is detectable in power subgroups, as is the $k$-Engel law for $k$ at most 4.
In contrast, detectability in power subgroups fails for solvability of given derived length: we construct a finite group $W$ such that $\ps{W}{2}$ and $\ps{W}{3}$ are metabelian but $W$ has derived length $3$.
We analyse the complexity of the detectability of commutativity in power subgroups, in terms of finite presentations that encode a proof of the result.
\end{abstract}

\maketitle


\section{Introduction}

This article studies the following broad question: what can be deduced about a group $G$ by examining its power subgroups $\ps{G}{n} = \gen{g^n : g \in G}$?
In particular, can one infer which laws $G$ satisfies?

Let $F_\infty = F(x_1, x_2, \dots)$ be the free group on the basis $\{x_1, x_2, \dots\}$.
A \emph{law} (or \emph{identity}) is a word $w \in F_\infty$, and we say a group $G$ \emph{satisfies} the law $w$ if $\varphi(w) = 1$ for all homomorphisms $\varphi \colon F_\infty \to G$.
For notational convenience, when we require only variables $x_1$ and $x_2$ we will instead write $x$ and $y$.
We can also think of a law $w$ on $k$ variables $x_1, \dots, x_k$ as a function $w \colon \underbrace{G \times \dots \times G}_{\text{$k$ times}} \to G$, written $w(g_1, \dots, g_k) := \varphi(w)$ for a homomorphism $\varphi \colon F_\infty \to G$ such that $\varphi(x_i) = g_i$.

Laws give a common framework for defining various group properties; basic examples include commutativity (corresponding to the law $[x, y]$), having exponent~$m$ (the Burnside law $x^m$), being metabelian (the law $[[x_1, x_2], [x_3, x_4]]$), and nilpotency of class at most $c$ (the law $[[[\dots[x_1, x_2], x_3], \dots, x_c], x_{c+1}]$).

\begin{defn}
    \label{defn:detectable}
    A group law $w$ is \emph{detectable in power subgroups} if, for all coprime $m$ and $n$, a group $G$ satisfies $w$ if and only if the power subgroups $\ps{G}{m}$ and $\ps{G}{n}$ both satisfy $w$.
\end{defn}

A subgroup of $G$ will satisfy all the laws of $G$, but in general it is possible even for coprime $m$ and $n$ that the power subgroups $\ps{G}{m}$ and $\ps{G}{n}$ satisfy a common law that $G$ does not; for example, the holomorph $G = \Zmod{7} \rtimes \Zmod{6}$ (where $\Zmod{6} \cong \operatorname{Aut}\Zmod{7}$ acts faithfully) was shown to have this property in \cite[Example 8.2]{neumann_cubed} (and is in fact the smallest such group).
A concrete example of a law that holds in $\ps{G}{2}$ and $\ps{G}{3}$ but not $G$ is $[[x^2, y^2]^3, y^3]$.
Another basic example is the holomorph $G = \Zmod{9} \rtimes \Zmod{6}$, which does not satisfy the law $[x^2, x^y]$ although $\ps{G}{2}$ and $\ps{G}{3}$ do.

\begin{eg}
    The law $x^r$ is detectable in power subgroups.
\end{eg}
This basic example is immediate: for every $g \in G$, if $(g^m)^r = 1$ and $(g^n)^r = 1$, then $g^r = 1$ as $m$ and $n$ are coprime.

A classical theme in group theory is the study of conditions that imply that a group is abelian.
This was recently revived by Venkataraman in \cite{gv}, where she proved that commutativity is detectable in power subgroups for finite groups.
We can extend this to infinite groups using residual finiteness of metabelian groups (a theorem of P.~Hall \cite[15.4.1]{robinson}); it appears that this result is folklore.

In this article we prove that this result generalizes to the nilpotent case:
\begin{keycor}
    \label{cor:nilpotency_is_detectable}
    Let $m$ and $n$ be coprime and let $c \geq 1$.
    Then a group $G$ is nilpotent of class at most $c$ if and only if $\ps{G}{m}$ and $\ps{G}{n}$ are both nilpotent of class at most $c$.
\end{keycor}
Fitting's Theorem (see \ref{thm:fitting} below) readily implies a weak form of the ``if'' direction, namely that $G$ is nilpotent of class at most $2c$, but it is much less obvious that the precise nilpotency class is preserved.

Detectability of laws in power subgroups has an elegant formulation in the language of group varieties, which we develop in Section~\ref{subsection:varieties}.
The reader unfamiliar with varieties should not be deterred, as our use of this language is simply a means of expressing our reasoning in a natural and general setting.
In particular, our treatment of varieties is essentially self-contained, and no deep theorems are called upon.

Let $\cN_c$ denote the variety of nilpotent groups of class at most $c$ and let $\cB_m$ denote the `Burnside' variety of groups of exponent $m$.
Employing the notion of product varieties (Definition~\ref{defn:product}), we can restate the conclusion of Corollary~\ref{cor:nilpotency_is_detectable} as $\cN_c \cB_m \cap \cN_c \cB_n = \cN_c$.
We prove this as a corollary of a stronger result:

\begin{keythm}
    \label{thm:any_nilpotent_variety}
    Let $\cV$ be a locally nilpotent variety and let $m$ and $n$ be coprime.
    Then \[
        \cV \cB_m \cap \cV \cB_n = \cV.
    \]
\end{keythm}

A variety is \emph{locally nilpotent} if its finitely generated groups are nilpotent, or, equivalently, if its groups are locally nilpotent.
A topic with a rich history, dating to work of Burnside, is that of Engel laws.
The $k$-Engel law is defined recursively by $E_0(x, y) = x$ and $E_{k+1}(x, y) = [E_k(x,y), y]$.
For example, the $3$-Engel law is $[[[x, y], y], y]$.
Havas and Vaughan-Lee \cite{havasvaughanlee} proved local nilpotency for $4$-Engel groups, so we have the following:
\begin{keycor}
    \label{cor:Engel}
    Let $m$ and $n$ be coprime and let $k \leq 4$.
    A group $G$ is $k$-Engel if and only if $\ps{G}{m}$ and $\ps{G}{n}$ are both $k$-Engel.
\end{keycor}
It is an open question whether a $k$-Engel group must be locally nilpotent for $k \geq 5$.
Recently A.~Juhasz and E.~Rips have announced that it does not have to be locally nilpotent for sufficiently large~$k$.

The class of virtually nilpotent groups plays an important role in geometric group theory, dating back to Gromov's seminal Polynomial Growth Theorem: a finitely generated group is virtually nilpotent if and only if it has polynomial growth \cite{gromov_poly}.
Because of this prominence, we also prove that virtual nilpotency is detectable in power subgroups (Corollary~\ref{cor:virtual_nilpotency}).

In contrast, solvability of a given derived length is not detectable in power subgroups; this fails immediately and in a strong sense as soon as we move beyond derived length one, that is, beyond abelian groups.

\begin{keythm}
    \label{thm:metabelian_counterexample}
    Let $\cM$ denote the variety of metabelian groups.
    Then \[
        \cM \cB_2 \cap \cM \cB_3 \neq \cM.
    \] Indeed, there exists a finite group $W$ such that $\ps{W}{2}$ and $\ps{W}{3}$ are both metabelian but $W$ is of derived length 3.
\end{keythm}

The construction of $W$ is rather involved and ad hoc, and does not have an obvious generalization.
The smallest such $W$ has order 1458.

This is yet another example of the chasm between nilpotency and solvability. Other properties that we lose when crossing from finitely generated nilpotent groups to finitely generated solvable groups include the following: residual finiteness, solvability of the word problem, polynomial growth, and finite presentability of the relatively free group.

As the free nilpotent group of class $c$ is finitely presented, we know \emph{a priori} that Corollary~\ref{cor:nilpotency_is_detectable} will be true for fixed $m$ and $n$ if and only if it is provable in a very mechanical way, namely via a finite subpresentation of a canonical presentation for the free group of rank $c+1$ in the variety $\cN_c \cB_m \cap \cN_c \cB_n$ (in a way which we make precise in Section~\ref{subsection:general}).
Since such a finite presentation `proving' the theorem for those $m$ and $n$ exists, it is natural to ask what such a presentation looks like: what is the minimum number of relators needed, does that number depend on $m$ and $n$, and how must the specific relators change with $m$ and $n$?

We analyse in detail the abelian case, where the answer to all of these questions is: surprisingly little.
\begin{keythm}
    \label{thm:its_abelian}
    Let $m$ and $n$ be coprime. The following is a presentation of $\mathbb{Z} \times \mathbb{Z}$ \[
        \gp{ a, b }{ [a^m, b^m], [a^m, (ab)^m], [b^m, (ab)^m], [a^n, b^n], [a^n, (ab)^n], [b^n, (ab)^n] }.
    \]
\end{keythm}

The structure of this article is as follows, and reflects the structure of the introduction we have just given.
In Section~\ref{section:basic_notions} we set up some basic theory.
We prove positive results, including Theorem~\ref{thm:any_nilpotent_variety}, in Section~\ref{section:nilpotent}.
We then prove the negative result Theorem~\ref{thm:metabelian_counterexample} in Section~\ref{section:solvable}.
The complexity analysis with Theorem~\ref{thm:its_abelian} follows in Section~\ref{section:complexity}, and finally we record some open problems in Section~\ref{section:problems}.

\section{Basic notions}
\label{section:basic_notions}

In this section we develop some basic tools which will be helpful, including aspects of the theory of group varieties.
We also probe the definition of detectability in power subgroups: why specifically power subgroups, and what about the non-coprime case?

For the first question, there are easy examples showing that we cannot in general determine if a group law is satisfied just by examining two arbitrary subgroups, even if they are assumed to be normal and to generate the whole group: it is essential that we examine the characteristic, ``verbally defined'' power subgroups.
For instance, the integral Heisenberg group $\gp{x,y,z}{[x,y]=z, [x,z] = [y, z] = 1}$ is the product of the two normal subgroups $\gen{x,z}$ and $\gen{y,z}$, which are both isomorphic to $\Z \times \Z$, however the whole group is not abelian.

We now turn to the question of coprimality.
For a property $\cP$ of groups, we say a group $G$ \emph{has $\cP$ coprime power subgroups} if there exist coprime $m$ and $n$ such that $\ps{G}{m}$ and $\ps{G}{n}$ both have the property $\cP$.
For example, using this terminology we can state the theorem of \cite{gv} as: a finite group with abelian coprime power subgroups is abelian.

\textbf{Notation.} We write conjugation $g^h = h^{-1} g h$ and commutator $[g, h] = g^{-1} h^{-1} g h$.

An easily proved property of power subgroups (and verbal subgroups in general, see below for the definition) is the following:
\begin{lem}
    \label{lem:basic_quotient}
    Let $\varphi \colon G \epi Q$ be a surjective group homomorphism and $m$ an integer.
    Then $\varphi(\ps{G}{m}) = \ps{Q}{m}$ and $\ps{G}{m} \leq \varphi^{-1}(\ps{Q}{m})$.
    \qed
\end{lem}

Power subgroups pick up torsion elements:
\begin{lem}
    \label{lem:super_basic}
    Suppose that $g$ has finite order $r$ coprime to $m$.
    Then $g \in \gen{g^m}$.
\end{lem}

\begin{proof}
    There exist integers $x$ and $y$ such that $xr + ym = 1$.
    Now \[ g = g^{xr + ym} = (g^r)^x (g^m)^y = (g^m)^y. \qedhere\]
\end{proof}

\subsection{Varieties}
\label{subsection:varieties}

We give a self-contained treatment of some basics from the theory of varieties of groups.
For further details, the reader is referred to Hanna Neumann's classic book \cite{h_neumann_book}.

\begin{defn}[Variety of groups]
    A \emph{variety} of groups is the class of all groups satisfying each one of a (possibly infinite) set of laws.
\end{defn}

\begin{eg}
    Corresponding to the examples of laws given above (immediately before Definition~\ref{defn:detectable}), we have the following examples of varieties:
    \begin{itemize}
        \item $\cA$ -- the variety of abelian groups
        \item $\cB_m$ -- the `Burnside' variety of groups of exponent $m$ (or exponent dividing $m$, depending on the definition of exponent used)
        \item $\cM$ -- the variety of metabelian groups.
        \item $\cN_c$ -- the variety of nilpotent groups of nilpotency class at most $c$
    \end{itemize}
\end{eg}

\begin{prop}
    A variety is closed under the operations of taking subgroups, quotients, and arbitrary Cartesian products.
    \qed
\end{prop}

In fact, \emph{every} class of groups which is closed under these operations is a variety (see \cite[15.51]{h_neumann_book}).

\begin{defn}[Product variety]
    \label{defn:product}
    Let $\cU$ and $\cV$ be varieties of groups.
    We define the \emph{product variety} $\cU \cV$ to be the class of groups which are an extension of a group from $\cU$ by a group from $\cV$.
    That is, $G \in \cU \cV$ if there exists $N \normal G$ such that $N \in \cU$ and $G / N \in \cV$.
    We define the product of two classes of groups similarly.
\end{defn}

\begin{eg}
    Let $\cA$ and $\cM$ denote the varieties of abelian and metabelian groups, respectively. Then $\cM = \cA \cA$.
\end{eg}

We check that the product variety is indeed a variety as follows.
Let $\cV(G) \leq G$ denote the \emph{verbal subgroup} of $G$ corresponding to $\cV$, that is, the subgroup generated by the images of the defining laws of $\cV$ under all maps $F_\infty \to G$.
Thus $G \in \cV$ if and only if the verbal subgroup $\cV(G) = 1$.
As defining laws for $\cU \cV$ we take the images of the defining laws of $\cU$ under all maps $F_\infty \to \cV(F_\infty)$.
Let $G$ be a group and suppose $N \normal G$, with $q \colon G \to G / N$ the natural homomorphism.
Then $\cV(G/N) = q(\cV(G))$ (cf.\@\xspace~Lemma~\ref{lem:basic_quotient}), so the quotient $G / N$ is in $\cV$ if and only if $\cV(G) \leq N$.
Thus $G \in \cU \cV$ if and only if $N = \cV(G)$ is in the variety $\cU$.
Every map $F_\infty \to \cV(G)$ factors through some map $F_\infty \to \cV(F_\infty)$, so we see that $\cV(G) \in \cU$ if and only if it satisfies every law which is the image of a defining law of $\cU$ in $\cV(F_\infty)$.
We will explore this further in Section~\ref{section:complexity}

\begin{prop}[{\cite[Theorem 21.51]{h_neumann_book}}]
    The product of varieties of groups is associative.
\end{prop}

Thus the varieties form a monoid under product, and the unit is the variety $\mathbb{1}$ consisting of only the trivial group.
We introduce a more restrictive notion of product for two classes of groups.
\begin{defn}[Normal product class]
    \label{defn:odot}
    Let $\cC$ and $\cD$ be classes of groups.
    We define the \emph{normal product} of $\cC$ and $\cD$, denoted $\cC \nprod \cD$, to be the class of groups $G$ with normal subgroups $C \in \cC$ and $D \in \cD$ such that $G = CD$.
\end{defn}

In particular, $\cD \nprod \cC = \cC \nprod \cD \subseteq \cC \cD \cap \cD \cC$.
This last inclusion can be proper, for example, $\cA \nprod \cA \subset \cN_2$ (by Theorem~\ref{thm:fitting}, due to Fitting), whereas $\cA \cA = \cM$.

\begin{prop}
    \label{prop:varietal_characterization}
    Let $G$ be a group, $\cV$ a variety of groups, and $m$ an integer.
    Recall that $\cB_m$ denotes the Burnside variety of exponent $m$.
    The power subgroup $\ps{G}{m} \in \cV$ if and only if $G \in \cV \cB_m$.
\end{prop}

\begin{proof}
    As in the proof that a product variety is a variety, we see that $G \in \cV \cB_m$ if and only if the verbal subgroup $\cB_m(G) \in \cV$, and $\cB_m(G) = \ps{G}{m}$.
\end{proof}

With this proposition in hand, we define a variety $\cV$ to be \emph{detectable in power subgroups} if, for all coprime $m$ and $n$, we have $\cV \cB_m \cap \cV \cB_n = \cV$ (the intersection of varieties is simply the intersection as classes of groups).
In this article, we mostly encounter varieties that are finitely based, that is, that can be defined by finitely many laws, and thus by a single law (the concatenation of these laws written in distinct variables $x_i$); in this case, detectability of the variety is simply detectability of such a single defining law.
It will be useful for us to understand how taking products of varieties interacts with taking intersection.
Although we do not have left-distributivity, we do have some upper and lower bounds, as the next proposition indicates.

\begin{prop}
    \label{prop:upper_and_lower}
    For all varieties $\cU$, $\cV$, $\cW$ we have \[
        \cU (\cV \cap \cW) \leq \cU \cV \cap \cU \cW \subseteq (\cU \nprod \cU) (\cV \cap \cW).
    \]
\end{prop}

(We write $\subseteq$ as the last term is not a variety in general.)

\begin{proof}
    The first inclusion is immediate, as $\cV \cap \cW \leq \cV$ implies that $\cU (\cV \cap \cW)$ is contained in $\cU \cV$, and similarly in $\cU \cW$.

    Now suppose that $G \in \cU \cV \cap \cU \cW$.
    This means $G$ has normal subgroups $N_\cV, N_\cW \in \cU$ such that $G / N_\cV \in \cV$ and $G / N_\cW \in \cW$.
    Let $N = N_\cV N_\cW \normal G$.
    The group $G / N$ will be a common quotient of $G / N_\cV$ and $G / N_\cW$, and thus in $\cV \cap \cW$, as varieties are closed under taking quotients. The kernel $N_\cV N_\cW$ is then a product of normal subgroups in $\cU$, so it is in the class $\cU \nprod \cU$.
\end{proof}

\begin{cor}
    \label{cor:metaU}
    Let $m$ and $n$ be coprime. Then for every variety $\cU$, \[
        \cU \cB_m \cap \cU \cB_n \leq \cU \nprod \cU.
    \]
\end{cor}

\begin{proof}
    Set $\cV = \cB_m$, $\cW = \cB_n$ in Proposition~\ref{prop:upper_and_lower} and note $\cB_m \cap \cB_n = \cB_{\operatorname{gcd}(m,n)} = \mathbb{1}$.
\end{proof}

In contrast, we do have right-distributivity of product of varieties over intersection:
\begin{prop}
    \label{prop:on_your_right}
    For all varieties $\cU, \cV, \cW$ we have \[
        (\cU \cap \cV) \cW = \cU \cW \cap \cV \cW.
    \]
\end{prop}

\begin{proof}
    The inclusion ``$\leq$'' is immediate.

    Suppose $G \in \cU \cW \cap \cV \cW$, with $N_\cU, N_\cV \normal G$ such that $N_\cU$ is in $\cU$, $N_\cV$ is in $\cV$, and both quotients $G/N_\cU, G/N_\cV$ are in $\cW$.
    We have a (generally non-surjective map) \[G \to G/N_\cU \times G/N_\cV\] with kernel $N_\cU \cap N_\cV$.
    That is, the kernel is in $\cU \cap \cV$, and the quotient is in $\cW$, since a variety is closed under Cartesian product and subgroups.
\end{proof}

Varieties are determined by their finitely generated groups:
\begin{prop}
    \label{prop:fg_subclass}
    Let $\cU$ and $\cV$ be varieties.
    Let $\cU_f$ denote the subclass of groups $G \in \cU$ such that $G$ is finitely generated and define $\cV_f$ similarly.
    Then $\cU = \cV$ if and only if $\cU_f = \cV_f$.
\end{prop}

\begin{proof}
Clearly $\cU = \cV$ implies $\cU_f = \cV_f$.
Suppose $\cV$ is \emph{not} contained in $\cU$, so there is a law~$w \in F_\infty$ which is satisfied by all groups in $\cU$, but there is some $G \in \cV$ and $\varphi \colon F_\infty \to G$ with $\varphi(w) \neq 1$.
The law $w$ is a word on finitely many letters $x_1, \dots, x_n$ in the basis for $F_\infty$, and we can assume $\varphi(x_i) = 1$ for all $i > n$.
The subgroup $G_0 \leq G$ generated by $\varphi(x_1), \ldots, \varphi(x_n)$ is an element of $\cV_f$.
We can consider $\varphi$ as a map $F_\infty \to G_0$ and so $G_0$ does not satisfy the law $w$.
Thus $\cV_f$ is not contained in $\cU$, so in particular $\cV_f$ is not contained in $\cU_f$.
\end{proof}

Recall a well-known fact about torsion groups, which we will apply several times.
\begin{prop}[{\cite[5.4.11]{robinson}}]
    \label{prop:solvable_torsion}
    A finitely generated solvable torsion group is finite.
\end{prop}

\subsection{Coprimality}
\label{subsection:coprimality}

The notion of detectability of a law in power subgroups does not make sense in general if one allows $m$ and $n$ for which $\operatorname{gcd}(m,n) = d > 1$; this could only say something about $\ps{G}{d} \leq \ps{G}{m} \cap \ps{G}{n}$ and not the whole group $G$.
For example, in a group $G$ of exponent $d$ the power subgroups are trivial and satisfy all laws $w \in F_\infty$, whereas $G$ does not if it is non-trivial.
A more extreme example is provided by the free Burnside groups of exponent $d$ for large odd $d$, which are infinite by the celebrated work of Novikov and Adian, and thus are not even solvable (by Proposition~\ref{prop:solvable_torsion}).

A precise formulation of this idea is the following:
\begin{prop}
    \label{prop:non_coprime_version}
    For every variety $\cV$ and for all integers $m$ and $n$, we have
    \begin{equation}
        \tag{$\star$}
        \label{eqn:non_coprime_version}
        \cV \cB_m \cap \cV \cB_n \geq \cV \cB_{\operatorname{gcd}(m,n)}.
    \end{equation}
    Suppose further that $\cV$ is detectable in power subgroups, that is, we have equality in \eqref{eqn:non_coprime_version} for the case of \emph{coprime} $m$ and $n$.
    Then we have equality in \eqref{eqn:non_coprime_version} for \emph{all} $m$ and $n$.
\end{prop}

\begin{proof}
    The inclusion \eqref{eqn:non_coprime_version} is immediate, as both $\cB_m$ and $\cB_n$ contain $\cB_{\operatorname{gcd}(m,n)}$.

    Suppose now that $\cV$ is detectable in power subgroups, and let $d = \operatorname{gcd}(m,n)$, $m' = m/d$, $n' = n/d$, so that $m'$ and $n'$ are coprime.
    We have $\cB_m \leq \cB_{m'} \cB_d$ (in general this inclusion may be strict), and similarly for $\cB_n$, and thus \[
        \cV \cB_m \cap \cV \cB_n \leq \cV \cB_{m'} \cB_d \cap \cV \cB_{n'} \cB_d = (\cV \cB_{m'} \cap \cV \cB_{n'}) \cB_d
    \] via right-distributivity of the product over intersection (Proposition~\ref{prop:on_your_right}), and implicitly using associativity of the variety product. By assumption of detectability, this last term is just $\cV \cB_d$.
\end{proof}

The reader is referred to \cite{boatman_olshanskii} for more on the fascinating topic of products of Burnside varieties.

\begin{rmk}[More than two powers]
    The notion of detectability is unchanged if we replace the two powers $m$ and $n$ with powers $m_1, m_2, \dots, m_k$ that are  mutually (not necessarily pairwise) coprime, that is, $\operatorname{gcd}(m_1, m_2, \dots, m_k) = 1$.
    This follows by an easy induction, which can be expressed conveniently using the characterization of Proposition~\ref{prop:non_coprime_version}.
\end{rmk}

\section{Locally nilpotent varieties are detectable}
\label{section:nilpotent}

The starting point for this section is a desire to generalize the result that commutativity is detectable in power subgroups to the nilpotent case.
For instance, can power subgroups detect whether a group is nilpotent of class at most 2?
We are carried quite a way towards our goal by Fitting's Theorem.

\begin{thm}[{Fitting, \cite[5.2.8]{robinson}}]
    \label{thm:fitting}
    Let $M$ and $N$ be normal nilpotent subgroups of a group $G$. If $c$ and $d$ are the nilpotency classes of $M$ and $N$, then $L = MN$ is nilpotent of class at most $c + d$.
\end{thm}

However, this will only tell us, for instance, that if the power subgroups are nilpotent of class at most 2, then our group of interest is nilpotent of class at most 4.
We first lay some foundations towards proving the general Theorem~\ref{thm:any_nilpotent_variety}, then see in Theorem~\ref{thm:nilpotent_laws} how we can reduce the bound of $2c$ to $c$, as in Corollary~\ref{cor:nilpotency_is_detectable}.
By proving the general theorem, we will also be able to conclude that certain Engel laws are detectable.

\begin{prop}
    \label{prop:still_got_it}
    Let $m$ and $n$ be coprime.
    Let $\cC$ denote the class of nilpotent, locally nilpotent, solvable, or locally solvable groups.
    Then $\cC \cB_m \cap \cC \cB_n = \cC$.
\end{prop}

\begin{proof}
    In each of the first three cases, this is an application of Corollary~\ref{cor:metaU} together with the corresponding standard result that the appropriate $\cC \nprod \cC$ is equal to $\cC$: Fitting's Theorem for the nilpotent case, the Hirsch--Plotkin Theorem \cite[12.1.2]{robinson} for the locally nilpotent case, and the solvable case is elementary.
    (This in fact shows the result still holds after replacing $\cB_m$ and $\cB_n$ with two arbitrary varieties with trivial intersection.)
    However, a group which is the product of two normal locally solvable subgroups need not be locally solvable, as shown by P.~Hall \cite[Theorem 8.19.1 (i)]{robinson_soluble_2}, so for the remaining case we exploit the power subgroup structure.
    This argument also allows us to conclude the locally nilpotent case from Fitting's theorem, without the need to invoke Hirsch--Plotkin.

    Assume now that $G \in \cC \cB_m \cap \cC \cB_n$ is finitely generated, so that its quotient $G / \ps{G}{m}$ is finitely generated and of exponent $m$.
    By the second isomorphism theorem, \[
        G / \ps{G}{m} \cong \ps{G}{n} / (\ps{G}{m} \cap \ps{G}{n})
    \] and so since $\ps{G}{n}$ is locally solvable, its finitely generated quotient $G / \ps{G}{m}$ is solvable.
    Now $G / \ps{G}{m}$ is a finitely generated solvable torsion group, and thus finite (Proposition~\ref{prop:solvable_torsion}).
    Hence the subgroup $\ps{G}{m} \normal G$ is of finite index, so it is finitely generated, and since groups in $\cU$ are locally solvable, $\ps{G}{m}$ is in fact solvable.
    Similarly, $\ps{G}{n}$ is solvable.
    Thus $G = \ps{G}{m} \ps{G}{n}$ is solvable.
\end{proof}

\begin{thm}
    \label{thm:nilpotent_laws}
    Let $G$ be a finitely generated nilpotent group and let $m$ and $n$ be coprime.
    If $\ps{G}{m}$ and $\ps{G}{n}$ both satisfy a law $w$, then $G$ satisfies $w$.
\end{thm}

In other words, the variety generated by $G$ is the intersection of the varieties generated by $\ps{G}{m}$ and $\ps{G}{n}$.
(The variety generated by a group is the intersection all varieties containing it.)

\begin{proof}
    Suppose for the sake of contradiction that there is a homomorphism $\varphi \colon F_\infty \to G$ with $\varphi(w) \neq 1$.
    Since $G$ is finitely generated and nilpotent, it is residually finite \cite[5.4.17]{robinson}, so there is a map $q \colon G \epi Q$ for some finite group $Q$ such that $q(\varphi(w)) \neq 1$.
    As $G$ is nilpotent, so is $Q$, and thus $Q$ is the direct product of its Sylow subgroups \cite[5.2.4]{robinson}.
    We compose $q$ with a projection onto a Sylow subgroup in which $q(\varphi(w))$ has non-trivial image, to get $q_p \colon G \epi Q_p$.
    Without loss of generality, $p$ does not divide $m$ so that $\ps{Q_p}{m} = Q_p$ (Lemma~\ref{lem:super_basic}).
    This gives a contradiction, as $\ps{Q_p}{m} = q_p(\ps{G}{m})$ (Lemma~\ref{lem:basic_quotient}), and $\ps{G}{m}$ satisfies the law $w$.
\end{proof}

\begin{thmLocallyNilpotentVariety}
    Let $\cV$ be a locally nilpotent variety and let $m$ and $n$ be coprime.
    Then \[
        \cV \cB_m \cap \cV \cB_n = \cV.
    \]
\end{thmLocallyNilpotentVariety}

\begin{proof}
    By Proposition~\ref{prop:fg_subclass}, it suffices to consider finitely generated $G \in \cV \cB_m \cap \cV \cB_n$.
    Since $\cV$ is locally nilpotent, Proposition~\ref{prop:still_got_it} guarantees that $G$ is locally nilpotent.
    As $G$ is in fact finitely generated, we can now apply Theorem~\ref{thm:nilpotent_laws} to conclude that $G$ satisfies every law which holds in both $\ps{G}{m}$ and $\ps{G}{n}$.
    Since $\ps{G}{m}$ and $\ps{G}{n}$ are in the variety $\cV$, we conclude that $G \in \cV$.
\end{proof}

The nilpotent groups of class at most $c$ form the variety $\cN_c$, so the following corollary is immediate.

\begin{corNilpotencyClass}
    Let $m$ and $n$ be coprime and let $c \geq 1$.
    Then a group $G$ is nilpotent of class at most $c$ if and only if $\ps{G}{m}$ and $\ps{G}{n}$ are both nilpotent of class at most $c$.
    \qed
\end{corNilpotencyClass}

This means that the precise nilpotency class of $G$ is the maximum of the precise nilpotency classes of $\ps{G}{m}$ and $\ps{G}{n}$.

\begin{corEngel}
    Let $m$ and $n$ be coprime and let $k \leq 4$.
    A group $G$ is $k$-Engel if and only if $\ps{G}{m}$ and $\ps{G}{n}$ are both $k$-Engel.
\end{corEngel}

\begin{proof}
    The variety of $4$-Engel groups was shown to be locally nilpotent by Havas and Vaughan-Lee \cite{havasvaughanlee}.
    This also implies the (previously known) $k \leq 3$ cases, as it is clear from the definition $E_{k+1}(x,y) = [E_k(x, y), y]$ that a $k$-Engel group is also $(k+1)$-Engel.
\end{proof}

\begin{rmk}
    Gruenberg proved that a locally solvable $k$-Engel group is locally nilpotent \cite{gruenberg}, so the generality achieved in Proposition~\ref{prop:still_got_it} would not help to establish detectability of a Engel law beyond the locally nilpotent case.
    For a survey on Engel groups, the reader is referred to \cite{traustason}.
\end{rmk}

For the sake of completeness, and motivated by its importance in geometric group theory, we show that the class of virtually nilpotent groups (groups with a nilpotent subgroup of finite index) is detectable in power subgroups.
We first prove a more general result, and then use the structure of subgroups as specifically power subgroups to argue that the precise `virtual nilpotency class' is preserved.

\begin{prop}
    \label{prop:virtually_nilpotent}
    Let $G$ be a group with normal subgroups $A$ and $B$ such that $G = AB$.
    Suppose that $A$ and $B$ are virtually nilpotent.
    Then $G$ is virtually nilpotent.
\end{prop}

\begin{proof}
    To invoke Fitting's Theorem, we require nilpotent subgroups that are normal in~$G$.
    Let $A_0$ be the normal subgroup of $A$ which is nilpotent and of minimal finite index.
    Such an $A_0$ exists as $A$ is virtually nilpotent, and it is unique by Fitting's Theorem, as the product of two finite index normal nilpotent subgroups of $A$ is then a normal nilpotent subgroup of smaller index.
    ($A_0$ is the `nilpotent radical' or `Hirsch--Plotkin radical' of $A$.)
    Now $A_0$ is characteristic in $A$, and thus normal in $G$.
    We define $B_0$ similarly.

    As $A_0$ and $B_0$ are both nilpotent and normal in $G$, their product $A_0 B_0$ is nilpotent.
    Since $A_0$ and $B_0$ are finite index in~$A$ and~$B$ respectively, and normal in~$G$, we conclude that $A_0 B_0$ is finite index in $AB = G$.
    That is, $G$ is virtually nilpotent.
\end{proof}

\begin{cor}
    \label{cor:virtual_nilpotency}
    Let $G$ be a finitely generated group and let $m$ and $n$ be coprime.
    If $\ps{G}{m}$ and $\ps{G}{n}$ both have finite index subgroups which are nilpotent of class at most $c$, then so does $G$.
\end{cor}

\begin{proof}
    By Proposition~\ref{prop:virtually_nilpotent}, $G$ is virtually nilpotent.
    Thus $G / \ps{G}{m}$ has a finite index subgroup which is nilpotent, and moreover finitely generated and of exponent $m$, and hence finite (by Proposition~\ref{prop:solvable_torsion}).
    Now $\ps{G}{m}$ is finite index in $G$, so its finite index subgroups are finite index in $G$.
\end{proof}

\section{Derived length is not detectable}
\label{section:solvable}

In this section we show by explicit example that one cannot extend the above results for the nilpotent case to the solvable case.
Of course, Proposition~\ref{prop:still_got_it} tells us that a group with solvable coprime power subgroups is itself solvable: the class of solvable groups is closed under extensions.
The point is that we do not have the precise control over derived length which we did for nilpotency class.

\begin{thmMetabelian}
    Let $\cM$ denote the variety of metabelian groups.
    Then \[
        \cM \cB_2 \cap \cM \cB_3 \neq \cM.
    \] Indeed, there exists a finite group $W$ such that $\ps{W}{2}$ and $\ps{W}{3}$ are both metabelian but $W$ is of derived length 3.
\end{thmMetabelian}

\begin{proof}
    Let $H_3$ denote the mod-$3$ Heisenberg group, which is the non-abelian group of order 27 and exponent 3.
    It admits the presentation \[
        H_3 = \gp{ x, y, z }{ x^3 = y^3 = 1, z = [x, y], [x, z] = [y, z] = 1 }.
    \] Write $\Zmod{n}$ for the cyclic group of order $n$.
    The group $W$ is constructed as \[
        W := (\Zmod{9} \times \Zmod{9}) \rtimes_\varphi (H_3 \times \Zmod{2})
    \] where, letting $\Zmod{2} = \gen{t}$, the action is defined by $\varphi \colon H_3 \to \SL_2(\Zmod{9})$ which maps \[
        x \mapsto \mat{1}{-1}{3}{-2}, \quad y \mapsto \mat{-2}{0}{0}{4}, \quad t \mapsto \mat{-1}{0}{0}{-1}.
    \]

    We check that this is a well-defined group action.
    (It is in fact faithful, however this is -- while easily verified -- unnecessary for the proof.)
    Let $X := \varphi(x)$, $Y := \varphi(y)$, $T := \varphi(t)$, and $Z := [X, Y] = \varphi(z)$.
    As $T$ is order 2 and central in $\SL_2 (\Zmod{9})$, we only need to check the map $H_3 \to \SL_2 (\Zmod{9})$.
    We see first that \[
        X^2 = \mat{-2}{1}{-3}{1}, \quad Y^2 = \mat{4}{0}{0}{-2}
    \] and thus $X^3 = Y^3 = 1$. As \[
        X Y = \mat{-2}{-4}{-6}{-8} = \mat{-2}{-4}{3}{1}
    \] and \[
        X^{-1} Y^{-1} = \mat{-2}{1}{-3}{1} \mat{4}{0}{0}{-2} = \mat{-8}{-2}{-12}{-2} = \mat{1}{-2}{-3}{-2}
    \] we see that \[
        Z = \mat{1}{-2}{-3}{-2} \mat{-2}{-4}{3}{1} = \mat{-8}{-6}{0}{10} = \mat{1}{3}{0}{1}
    \] so that $Z$ has order $3$, with $Z^{-1} = \mat{1}{-3}{0}{1}$.

    We now determine the conjugation action of $Z$: \[
        Z^{-1} \mat{a}{b}{c}{d} Z = \mat{a-3c}{b-3d}{c}{d} Z = \mat{a-3c}{b + 3(a-d)}{c}{d+3c}
    \] as $9c = 0$, so the matrix $\mat{a}{b}{c}{d}$ is in the centralizer of $Z$ precisely when both $a-d$ and $c$ lie in $3 \Zmod{9}$.
    This is true for both $X$ and $Y$, so we have $[X,Z] = [Y,Z] = 1$ as required, which completes the verification that $\varphi$ is a well-defined group homomorphism (or in other words, $X$, $Y$ and $T$ generate a subgroup of $\SL_2(\Zmod{9})$ isomorphic to $H_3 \times \Zmod{2}$, modulo the injectivity of $\varphi$ which we will not verify here).

    We claim that $\ps{W}{2} = (\Zmod{9} \times \Zmod{9}) \rtimes H_3$ and $\ps{W}{3} = (\Zmod{9} \times \Zmod{9}) \rtimes \Zmod{2}$.
    The ``$\leq$'' inclusion is Lemma~\ref{lem:basic_quotient}, and the other inclusion is not necessary for the proof and so is left to the curious reader (if it were not the case, it would only make the task at hand easier).
    The group $\ps{W}{3}$ is obviously metabelian, as it is exhibited as the semidirect product of one abelian group and another abelian group.
    On the other hand, $\ps{W}{2}$ will require the following basic computations.

    Recall first the following:
    \begin{lem}
        Let $G = N \rtimes K$.
        Then the derived subgroup $G' = (N' [N, K]) \rtimes K'$.
    \end{lem}

    One can prove the lemma by verifying that every commutator in $G$ lies in the subgroup generated by $N'$, $[N, K]$ and $K'$, and then noting that the action of $K$ on $N$ restricts to an action on $N' [N,K]$.

    In the present case of $\ps{W}{2} = (\Zmod{9} \times \Zmod{9}) \rtimes H_3$, since $N = \Zmod{9} \times \Zmod{9}$ is abelian we simply have $[N, K] \rtimes K'$.
    The subgroup $[N, K]$ is generated by $(I-X)n$ and $(I-Y)n$ for $n \in \Zmod{9} \times \Zmod{9}$.
    As \[
        I - X = \mat{0}{1}{-3}{3}, \quad I - Y = \mat{3}{0}{0}{-3},
    \] we see that $[N, K] = \Zmod{9} \times 3 \Zmod{9} = \set{ \begin{pmatrix} u \\ 3v \end{pmatrix} }{ u,v \in \Zmod{9} }$.

    Now $H_3' = \gen{z} \cong \Zmod{3}$, and we see that the set of invariants for $Z$ is auspiciously none other than $\Zmod{9} \times 3 \Zmod{9}$.
    Thus $[N, K] \rtimes K'$ is abelian, so $\ps{W}{2}$ is metabelian.

    On the other hand, for the negation action of $\Zmod{2}$ on $N = \Zmod{9} \times \Zmod{9}$ we have $[N, \Zmod{2}] = 2N = N$, so we see that $W' = N \rtimes H_3'$, which has derived length 2, so $W$ has derived length $3$ as claimed.
\end{proof}

\begin{rmk}
    The group $W$ constructed above has order $4374 = 2 \times 3^7$.
    It is also possible to construct a group $W$ satisfying the requirements of the theorem as $(\Zmod{3} \times \Zmod{9}) \rtimes (H_3 \times \Zmod{2})$, of order $1458 = 2 \times 3^6$ (the action is \emph{not} simply a restriction or quotient of $\varphi$).
    An exhaustive search with \texttt{GAP} \cite{GAP4} revealed that this is in fact the smallest non-metabelian group with metabelian coprime power subgroups.
    (There are two such groups of order 1458, and their ID pairs in the Small Groups Library are (1458, 1178) and (1458, 1192).)
\end{rmk}

\begin{rmk}
    We cannot extend this construction in an obvious way from the case of $p = 3$ to other primes.
    In particular, it appears to depend on the existence of a matrix of order $p$ in $\SL_2(\Z)$, which only has torsion elements of order at most $6$.
\end{rmk}

\begin{rmk}
    One could ask for a finitely generated infinite group $W$ that shows that being metabelian is not detectable in power subgroups.
    However, the failure will still be only up to finite index: such a group is solvable, and thus its power subgroups are of finite index (as used in the proof of Proposition~\ref{prop:still_got_it}).
\end{rmk}

\section{Complexity analysis}
\label{section:complexity}

By complexity analysis, we are \emph{not} referring to analysis of algorithms and complexity classes such as $\textsf{P}$ and $\textsf{NP}$, but the flavour is similar: we wish to quantify the complexity of detectability of given laws (in a sense we shall make precise), and understand the asymptotic behaviour of this complexity when the powers $m$ and $n$ vary.

We can formulate detectability of commutativity using an infinitely presented group having the appropriate universal property.

\begin{prop}
The detectability of commutativity in power subgroups is equivalent to the fact that, for all coprime $m$ and $n$, the group $G_{m,n}$ defined by the infinite presentation \[
    G_{m,n} := \gp{a, b}{[u^m, v^m], [u^n, v^n] \enskip \Forall u, v \in F(a,b)}
\] is isomorphic to $\Z \times \Z$.
\end{prop}

\begin{proof}
If $G_{m,n}$ is non-abelian then it is a counterexample.
Suppose $G_{m,n}$ is abelian, and let $H$ be a group with $\ps{H}{m}, \ps{H}{n}$ abelian.
Then for every pair of elements $g,h \in H$ there is a homomorphism from $G_{m,n}$ to $\langle g, h \rangle \leq H$ defined by $a \mapsto g, b \mapsto h$, since all the relators have trivial image, and thus $g$ and $h$ commute.
Therefore $H$ is abelian.
Moreover, as all relators in the presentation of $G_{m,n}$ are a consequence of the commutativity of the two generators, $G_{m,n}$ is in fact isomorphic to the free abelian group on 2 generators, $\Z \times \Z$.
\end{proof}

Thus for coprime $m$ and $n$, the word $[a,b]$ is expressible in the free group as a product of conjugates of terms of the form $[u^m, v^m]$ and $[u^n, v^n]$.
Such an expression gives a proof that $\cA \cB_m \cap \cA \cB_n = \cA$.
At this point, it is natural to ask how many different such terms are needed to encode such a proof.
Before giving a very succinct choice of such terms, we phrase the set up in greater generality.

\subsection{General framework}
\label{subsection:general}

\begin{defn}
    Let $\gp{X}{R}$ be a group presentation.
    We call $\gp{X}{S}$ a \emph{subpresentation} of $\gp{X}{R}$ if $S \subseteq R$.
    If moreover $\ncls{S}_{F(X)} = \ncls{R}_{F(X)}$, then we call $\gp{X}{S}$ a \emph{core} of $\gp{X}{R}$.
\end{defn}
A core not only defines an isomorphic group: there is a natural isomorphism induced by the identity map on the free group $F(X)$.
We recall the following standard result.
\begin{lem}
    Let $\gp{X}{R}$ be a presentation of a group that admits a finite presentation, and assume that $X$ is finite.
    Then $\gp{X}{R}$ admits a finite core.
\end{lem}
Briefly, the proof is the following (see also \cite[Theorem 2.10]{chuck}).
Fix an isomorphism to the group defined by a finite presentation $\gp{Y}{S}$, with the isomorphism and its inverse induced by $\phi \colon F(X) \to F(Y)$ and $\psi \colon F(Y) \to F(X)$ respectively.
These data also give an isomorphism for the group $G$ defined by a subpresentation $\gp{X}{R'}$ provided that $x =_G \psi(\phi(x))$ for $x \in X$ and $\psi(s) =_G 1$ for $s \in S$.
These $\abs{X} + \abs{S}$ relations will each be a consequence of a finite subset $R$; the union of these gives us a finite choice for $R'$.

Let $\cV$ be a finitely based variety, endowed with a chosen finite set $\cL$ of defining laws.
We only need finitely many variables for the laws in $\cL$; suppose that each law is in $F_k \leq F_\infty$.
Suppose further that the relatively free group $F_k / \cV(F_k)$ is finitely presented (as an abstract group).
For example, if $\cV$ is nilpotent (or locally nilpotent) then this relatively free group is finitely generated and nilpotent, thus finitely presented.
Now $\cV \cB_m \cap \cV \cB_n$ is equal to $\cV$ if and only if its relatively free group of rank $k$ is (naturally) isomorphic to $F_k / \cV(F_k)$.
That is, if and only if the infinite presentation \[
    \cQ^\cL_{k,m,n} = \gp{ x_1, \dots, x_k }{ w(u_1, \dots, u_k), w(v_1, \dots, v_k) \text{ for } w \in \cL, u_i \in \ps{F_k}{m}, v_i \in \ps{F_k}{n} }
\] is a presentation for the finitely presented group $F_k / \cV(F_k)$.
If this is the case, then $\cQ = \cQ^\cL_{k,m,n}$ admits a finite core.
Thus there is a partial algorithm that will decide if $\cQ$ does indeed present $F_k / \cV(F_k)$: enumerate larger and larger finite subpresentations of $\cQ$ (that is, a filtration of the set of relators by finite sets) and attempt for longer and longer for each finite subpresentation to find a proof of isomorphism, proceeding in a diagonal fashion (we ``diagonalize'' the filtration and the isomorphism search).
In general, the isomorphism problem is \emph{partially} decidable (that is, there is an algorithm that will succeed in proving two input groups are isomorphic if they are isomorphic, but may fail otherwise), but our situation is easier, as we can fix the identity map on the generators (assuming our finite presentation for the relatively free group has $k$ generators).
So we only require a partial algorithm for the word problem: for instance, at the $r$-th attempt we can determine all words in $F_k$ which are a product of conjugates of at most $r$ relators or inverses of relators in the finite subpresentation, by words of length at most $r$, and freely reduce to see whether all the relators of the finite presentation for $F_k / \cV(F_k)$ appear.
Thus we have established the following:
\begin{prop}
    \label{prop:re}
    Let $\cV$ be a finitely based variety, and suppose that it admits a finite set of defining laws such that each law is on at most $k$ variables.
    Suppose that the relatively free group $F_k / \cV(F_k)$ is finitely presented as an abstract group.
    Then the set of coprime integers $m$ and $n$ for which \[
        \cV \cB_m \cap \cV \cB_n = \cV
    \] is a recursively enumerable set.
    That is, there is an algorithm which, given as input a pair $m, n$, will output YES and terminate in finite time if and only if the varieties are equal.
    \qed
\end{prop}

However, we have only demonstrated the existence of such an algorithm: to actually implement the algorithm, we require additionally a finite presentation of $F_k / \cV(F_k)$.
For example, a presentation of the free 2-generator 4-Engel group was obtained by Nickel \cite[\S 3.1]{nickel}.
To use Nickel's (polycyclic) presentation, where clearly only $a_1$ and $a_2$ are needed to generate the group, we use the obvious Tietze moves to remove the other generators; in general, given a finite presentation of $F_k / \cV(F_k)$ on more than $k$ generators, we could either enumerate presentations of the same group (in a blind search, via Tietze moves) until we construct one with $k$ generators, or replace the partial algorithm for the word problem with a partial algorithm for the isomorphism problem (thereby deferring the difficulty).

The argument establishing Proposition~\ref{prop:re} works in greater generality.
In fact, all that we used about the varieties $\cB_m$ and $\cB_n$ is that they admit a basis which is a recursively enumerable subset of $F_\infty$ (for example, a finite set).
This is what implies that $\ps{F_k}{m} = \cB_m(F_k)$ is recursively enumerable: for a general variety $\cU$, each element of $\cU(F_k)$ is a finite product of images of the defining laws, and each defining law has a recursively enumerable set of images in $F_k$.
Under these conditions, we have a corresponding recursively enumerable presentation $\cQ$ of the $k$-generated relatively free group in $\cV \cU \cap \cV \cW$.
Thus, for $\cV$ as in Proposition~\ref{prop:re}, there is a partial algorithm that takes as input the description of two recursively enumerable bases, for varieties $\cU$ and $\cW$, and will succeed in determining that $\cV \cU \cap \cV \cW = \cV$ when this is the case.
We do not need the assumption that we made in the ``Burnside'' case that the varieties $\cU$ and $\cW$ have trivial intersection, but if this were not the case we actually could have equality only if $\cV$ were the variety of all groups \cite[23.32]{h_neumann_book}.

\subsection{Complexity of the abelian case}

For the abelian case, where $\cV = \cA$ defined by the law $[x, y]$, $k = 2$, and $F_2 / \cA(F_2) = F_2 / [F_2, F_2] \cong \Z \times \Z$, there is an extraordinarily short finite presentation of $\cQ^{\{[x,y]\}}_{2,m,n}$.

\begin{thmAbelianPresentation}
    Let $m$ and $n$ be coprime. The following is a presentation of $\mathbb{Z} \times \mathbb{Z}$\,: \[
        \gp{ a, b }{ [a^m, b^m], [a^m, (ab)^m], [b^m, (ab)^m], [a^n, b^n], [a^n, (ab)^n], [b^n, (ab)^n] }.
    \]
\end{thmAbelianPresentation}

After proving this theorem, we became aware of another proof \cite{stackexchange} that groups with abelian power subgroups are abelian, from which one can extract a 2-generator 6-relator core of $\cQ$ which defines $\Z \times \Z$, just as in Theorem~\ref{thm:its_abelian}.
However, our proof has the advantage of uniformity in the words from the verbal subgroup used, whereas in the other proof the length of the words grows quadratically with $m$ and $n$.

The proof proceeds by first showing that the commutator $[a,b]$ is central; once we know this, the proof that it is trivial is very short.

Rather than prove that $G_{m,n}$ is nilpotent of class $2$ directly, we instead prove the stronger result that the group $\Gamma$ (defined below), an extension of $G_{m,n}$, is nilpotent of class $2$.
This group is moreover a common extension of all the $G_{m,n}$, so we see our introducing $\Gamma$ as abstracting away $m$ and $n$ from the proof.
(We will of course prove later that each $G_{m,n} \cong \Z \times \Z$, and so technically $\Z \times \Z$ itself is also a common extension \emph{a posteriori}, but we are constructing a group which is \emph{a priori} a common extension.)

\begin{defn}
    Let the group $\Gamma$ be defined by the presentation \[
        \gp{ a, b, x, y, z }{ [a, x], [b, y], [ab, z], [x, y], [x, z], [y, z], [ax, by], [ax, abz], [by, abz]}.
    \]
\end{defn}

\begin{lem}
    \label{lem:extension}
    The group $\Gamma$ is an extension of $G_{m, n}$, with the quotient map sending $a \mapsto a$ and $b \mapsto b$.
\end{lem}

\begin{proof}
    As $m$ and $n$ are coprime, there exist integers $p$ and $q$ such that $pm - qn = 1$, that is, $pm = qn + 1$.
    Define a map $\Gamma \to G_{m, n}$ by $a \mapsto a$, $b \mapsto b$, $x \mapsto a^{qn}, y \mapsto b^{qn}$ and $z \mapsto (ab)^{qn}$.
    This is easily checked to be well-defined, as every defining relator for $\Gamma$ is mapped to a relator of the form $[u^k, v^l]$ for some $u$ and $v$ with $[u, v]$ a defining relator of $G_{m, n}$ and $k, l \in \Z$.
\end{proof}

\begin{prop}
    \label{prop:nilp}
    The subgroup $\gen{a,b} \leq \Gamma$ is nilpotent of class $2$.
\end{prop}

\begin{rmk}
    \label{rmk:gamma_nilp}
    The group $\Gamma$ itself is nilpotent of class $2$, with $[\Gamma, \Gamma] \cong \Z$.
    However, we confine ourselves here to proving Proposition~\ref{prop:nilp}, which is all that is required for Theorem~\ref{thm:its_abelian}.
\end{rmk}

We prove Proposition~\ref{prop:nilp} in a sequence of lemmas.
It will be convenient to know that the symmetry in $a$ and $b$ of $G_{m, n}$ extends to $\Gamma$.

\begin{lem}
    \label{lem:symmetry}
    Let $\varphi : a \mapsto b \mapsto a, x \mapsto y \mapsto x, z \mapsto z^a$.
    Then $\varphi$ is an automorphism of $\Gamma$.
\end{lem}

\begin{proof}
    Since the above also defines an automorphism of the free group $F(a,b,x,y,z)$, it suffices to check that $\varphi$ is a well-defined group homomorphism.
    To verify this we now show that the images of the relators are trivial, in the cases where this is not immediate.
    Note that since $[ab, z] = 1$, we have $z^a = z^{b^{-1}}$.
    \begin{align*}
        \varphi([ab, z])   &= [ba, z^a]   = [ab, z]^a = 1 \\
        \varphi([x,z])     &= [y, z^a]    = [y, z^{b^{-1}}] = [y, z]^{b^{-1}} = 1 \\
        \varphi([y, z])    &= [x, z^a]    = [x, z]^a = 1 \\
        \varphi([ax, abz]) &= [by, baz^a] = [by, bza] = [by, zab]^{b^{-1}} = [by, abz]^{b^{-1}} = 1 \\
        \varphi([by, abz]) &= [ax, baz^a] =[ax, bza] = [ax, abz]^a = 1 \qedhere
    \end{align*}
\end{proof}

\begin{lem}
    \label{lem:ab_yx}
    The commutator $[ab, yx] = 1$ in the group $\Gamma$.
\end{lem}

\begin{proof}
    In light of Lemma~\ref{lem:symmetry}, we can instead prove $[ba, xy]=1$ as follows:
    \[
    \begin{array}{rl!{\quad \vline \quad}rl}
                & axybabz          &  &= abxz(ab)y \\
                &= (ax)(by)(abz)   &  &= abx(ab)zy \\
                &= (abz)(ax)(by)   &  &= abaxbyz \\
                &= abzxaby         &  &= abaxybz.
    \end{array}
    \]
    After cancelling on the left and right, we have $xyba = baxy$.
\end{proof}

\begin{lem}
    \label{lem:a_bzy}
    The commutator $[a, bzy] = 1$ in the group $\Gamma$.
\end{lem}

\begin{proof}
    We have
    \[
    \begin{array}{rl!{\quad \vline \quad}rl}
               & a(bzy)           &  &= by(ab)zb^{-1} \\
               &= (abz)(by)b^{-1} &  &= byz(ab)b^{-1} \\
               &= (by)(abz) b^{-1}&  &= (bzy)a.
    \end{array}
    \]
\end{proof}

\begin{lem}
    \label{lem:b_zax}
    The commutator $[b, zax] = 1$ in the group $\Gamma$.
\end{lem}

\begin{proof}
    This follows from Lemma~\ref{lem:a_bzy} by symmetry, as we have
    $\varphi([b, zax]) = [a, z^a b y] = [a, z^{(ab)^{-1} a} b y] = [a, z^{b^{-1}} b y] = [a, b z y] = 1$.
\end{proof}

In the following computations, we will frequently use the basic fact that if $[g, hk] = 1$ then $g^h = g^{k^{-1}}$.

\begin{lem}
    \label{lem:b_zm1x}
    The commutator $[b, z^{-1}x] = 1$ in the group $\Gamma$.
\end{lem}

\begin{proof}
    Note first that since $abz$ and $a$ both commute with $ax$, we have $[ax, bz] = 1$ and thus $(ax)^b = (ax)^{z^{-1}}$. Now
    \[
    \begin{array}{rl!{\vline}rll}
        abxb^{-1} &= (ab)(xy)y^{-1}b^{-1}                                          &&= (ax)^b & \\
                  &= (xy)(ab)y^{-1}b^{-1} \quad \text{Lemma~\ref{lem:ab_yx}}  &&= (ax)^{z^{-1}} \\
                  &= yxay^{-1}                                                     &&= (ab)^{z^{-1}} (b^{-1}x)^{z^{-1}} &\\
                  &= b^{-1} (by)(ax)y^{-1}                                         &&= ab z b^{-1} x z^{-1} &\\
                  &= b^{-1} (ax)(by)y^{-1}                                         &&= ab z b^{-1} z^{-1} x. &
    \end{array}
    \]
    Left-multiplying both sides by $z^{-1} b^{-1} a^{-1}$ gives $z^{-1} x b^{-1} = b^{-1} z^{-1} x$.
\end{proof}

\begin{lem}
    \label{lem:b_commaz}
    The commutator $[b, [a,z]] = 1$ in the group $\Gamma$.
\end{lem}

\begin{proof}
    By Lemma~\ref{lem:b_zax} we have $zax \in C_\Gamma(b)$, the centralizer of $b$, and by Lemma~\ref{lem:b_zm1x} we have $z^{-1}x \in C_\Gamma(b)$.
    The centralizer thus contains $z^{-1} x z a x = xax$, and thus also \[
        [xax, x^{-1} z] = [a, z]
    \] since $x$ is central in $\gen{a, x, z}$.
\end{proof}

We are now equipped to prove Proposition~\ref{prop:nilp}.

\begin{proof}[Proof of Proposition~\ref{prop:nilp}]

    Starting by applying Lemma~\ref{lem:ab_yx}, we see
    \[
    \begin{array}{rlrl!{\vline}rlrl}
        ab &= (yx) ab (yx)^{-1} &&                               &&= a^{bz} b^{za} && \text{Lemmas~\ref{lem:a_bzy}~and~\ref{lem:b_zax}} \\
           &= y x a y^{-1} b x^{-1} &&                              &&= (a^{bz} b^{za})^{z^{-1}} && \text{LHS $ab$ commutes with $z$} \\
           &= y a y^{-1} x b x^{-1} &&                              &&= a^b b^{zaz^{-1}} \\
           &= a^{y^{-1}} b^{x^{-1}} &&                              &&= a^b b^a && \text{Lemma~\ref{lem:b_commaz}}
    \end{array}
    \]
    Thus $b^a = (a^b)^{-1} ab = b^{-1} a^{-1} b a b = b^{-1} b^a b$.
    Since $b$ commutes with $b^a$, it commutes with $b^{-1} b^a = [b, a] = [a, b]^{-1}$.
    Applying $\varphi$, we see that also $[a, [a, b]] = 1$.
    Thus the subgroup $\gen{a,b} \leq \Gamma$ is nilpotent of class 2.
\end{proof}

\begin{proof}[Proof of Theorem~\ref{thm:its_abelian}]
    As $[a,b]$ is central in $\gen{a,b} \leq \Gamma$ (Proposition~\ref{prop:nilp}) and $\Gamma$ is an extension of $G_{m,n}$ (Lemma~\ref{lem:extension}), it follows that $[a,b]$ is central in $G_{m,n}$.
    Thus $[a^m, b^m] = [a, b]^{m^2}$ and $[a^n, b^n] = [a, b]^{n^2}$.
    Since $m$ and $n$ are coprime, so are $m^2$ and $n^2$.
    Now coprime powers of $[a,b]$ are both trivial, so $[a,b] = 1$.
\end{proof}

\begin{rmk}
    \label{rmk:place_an_order_for_your_orders}
    It is not sufficient to take 5 of the 6 relators of $G_{m, n}$ (for coprime $m, n > 1$).
    Suppose that we omitted $[b^n, (ab)^n]$ (the other cases are analogous).
    A folklore theorem states that for all integers $p, q, r > 1$ one can find a finite group containing elements $a$ and $b$ such that $a, b, ab$ have orders $p, q, r$ respectively (for a proof, see \cite[Theorem 1.64]{milne}).
    Such a group for $(p,q,r) = (n, m, m)$ will be a quotient of the group defined by our truncated presentation, as all defining commutativity relators hold trivially by virtue of one term having trivial image.
    It cannot be abelian, as otherwise the order of $ab$ would be $mn$.
\end{rmk}

\section{Open problems}
\label{section:problems}

While we have a surprisingly and uniformly small core (subpresentation) of $\cQ$ in the abelian case, we still do not know what is the minimum number of relators needed.

\begin{qn}
    \label{qn:four_rel}
    For which coprime $m$ and $n$ is there a $4$-relator presentation for $\Z \times \Z$ that encodes a proof that $\cA \cB_m \cap \cA \cB_n = \cA$, that is, when does the infinite presentation $\cQ^{\{[a,b]\}}_{2,m,n}$ have a $4$-relator core?
    In particular, when is \[
        \Delta_{m,n} = \gp{ a, b }{ [a^m, b^m], [a^n, b^n], [(ab)^m, (ba)^m], [(ab)^n, (ba)^n] }
    \] isomorphic to $\Z \times \Z$ ?
\end{qn}

The van Kampen diagram in Figure~\ref{fig:vk} proves that the question has a positive answer for $(m,n) = (2,3)$.
A van Kampen diagram \emph{per se} is purely topological; the geometry of the drawing has been chosen such that corners generally delimit the 4 subwords $u^{-1}$, $v^{-1}$, $u$ and $v$ in a commutator $[u,v]$.

\begin{figure}
    \centering
    \includegraphics[width=0.7\textwidth]{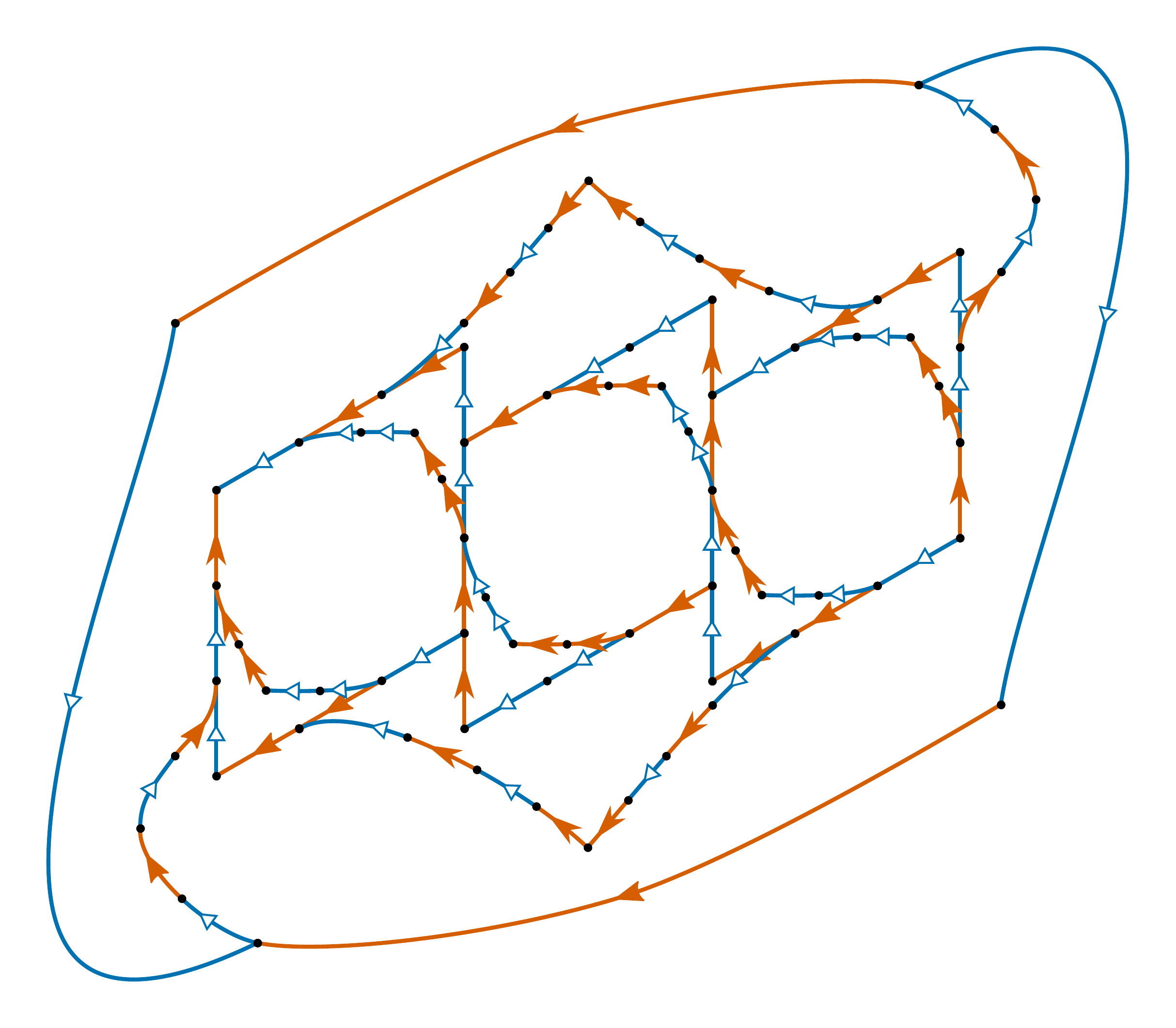}
    \caption{A van Kampen diagram proving that $[a,b] = 1$ in the group $\gp{ a, b }{ [a^2, b^2], [a^3, b^3], [(ab)^2, (ba)^2], [(ab)^3, (ba)^3] }$}
    \label{fig:vk}
\end{figure}

Computational experiments using \texttt{GAP} \cite{GAP4} and \texttt{magma} \cite{magma} provided some evidence for Question~\ref{qn:four_rel} having a positive answer.
In particular, we have verified this for $m = 2$ and odd $n < 50$.
However, we could not answer the question either way for $(m,n) = (3,4)$.
A necessary condition is for the four relators to generate the full \emph{relation module}, that is, the abelianization of the kernel $F(a,b)'$ of the full presentation, as a $\Z[\Z \times \Z]$-module.
(This is actually a cyclic module.)
This condition is met for $(m,n) = (3,4)$; in fact, it is met for all $m \leq 80$ and $n = m+1$.

Note that the argument of Remark~\ref{rmk:place_an_order_for_your_orders} cannot be used to show that the group $\Delta_{m,n}$ has a non-abelian finite quotient: if we ask that $(ab)^m = 1$ and $(ba)^n = 1$, then since $ab$ and $ba$ are conjugate we in fact have $ab = 1$.
Any possible non-abelian finite quotient is quite constrained; in particular, if the order of $ab$ is coprime to either $m$ or $n$ then we immediately have $[ab, ba] = 1$, and similarly if $a$ and $b$ both have order coprime to $m$ then $[a, b] = 1$.
Meanwhile, we have determined computationally for the $(m,n) = (3,4)$ case that if $a$ and $b$ have order dividing 24 then such a quotient is abelian.
This computation was performed using Holt's package \texttt{kbmag} \cite{kbmag1.5} to show that the commutator subgroup of $\Delta_{3,4} / \ncls{a^{24}, b^{24}}$, an \emph{a priori} finitely presented group, is trivial.

\begin{qn}
    Determine the analogous complexity for the nilpotency law \[\nu_c = [[[\dots[x_1, x_2], x_3], \dots, x_c], x_{c+1}].\]
    That is, how does the size of the smallest finite core of $\cQ^{\{\nu_c\}}_{c+1,m,n}$ vary with $c$, $m$ and $n$?
\end{qn}

It seems that the following classification problem would require substantial progress.
\begin{qn}
    Which laws are detectable in power subgroups?
\end{qn}
The difficulty is exemplified by the fact that the 4-Engel law is detectable, but it has been claimed that not all $k$-Engel laws imply the essential local nilpotency that we used.
We thus ask in particular:
\begin{qn}
    For which $k$ is the $k$-Engel law detectable in power subgroups.
\end{qn}
To summarize our knowledge at this time, $x^m$ is detectable in power subgroups, as is every locally nilpotent law (for example, the 4-Engel law $[x, y, y, y, y]$ or a nilpotency law such as $[[x_1, x_2], x_3]$).
On the other hand, $[[x_1, x_2], [x_3, x_4]]$ is not detectable, and neither are some assorted laws for which detectability also fails in finite groups, such as $[[x^2, y^2]^3, y^3]$ and $[x^2, x^y]$.

\subsection*{Acknowledgements}
I thank my supervisor Martin Bridson for guidance, encouragement, and helpful suggestions; Nikolay Nikolov for suggesting the problem of detectability of Engel laws; Peter Neumann for helpful conversations; Geetha Venkataraman for helpful correspondence; Robert Kropholler for helpful comments on a draft; and my officemates L. Alexander Betts and Claudio Llosa~Isenrich for helpful conversations.

\bibliographystyle{../transfer/halpha}
\bibliography{squaresandcubes}

\begin{thebibliography}{NNN62}

\bibitem[BCP97]{magma}
Wieb Bosma, John Cannon, and Catherine Playoust.
\newblock The {M}agma algebra system. {I}. {T}he user language.
\newblock {\em J. Symbolic Comput.}, 24(3-4):235--265, 1997.
\newblock Computational algebra and number theory (London, 1993).

\bibitem[BO15]{boatman_olshanskii}
Nicholas~S. Boatman and Alexander~Yu Olshanskii.
\newblock On identities in the products of group varieties.
\newblock {\em Internat. J. Algebra Comput.}, 25(3):531--540, 2015.

\bibitem[GAP16]{GAP4}
The GAP~Group.
\newblock {\em {GAP -- Groups, Algorithms, and Programming, Version 4.8.3}},
  2016.

\bibitem[Gro81]{gromov_poly}
Mikhael Gromov.
\newblock Groups of polynomial growth and expanding maps.
\newblock {\em Inst. Hautes \'Etudes Sci. Publ. Math.}, (53):53--73, 1981.

\bibitem[Gru53]{gruenberg}
K.~W. Gruenberg.
\newblock Two theorems on {E}ngel groups.
\newblock {\em Proc. Cambridge Philos. Soc.}, 49:377--380, 1953.

\bibitem[Hol09]{kbmag1.5}
D.~Holt.
\newblock \texttt{kbmag}, {K}nuth-{B}endix on {M}onoids and {A}utomatic
  {G}roups, {V}ersion 1.5.
\newblock \href {http://www.warwick.ac.uk/~mareg/kbmag}
  {\texttt{http://www.warwick.ac.uk/}\discretionary
  {}{}{}\texttt{\texttt{\symbol{126}}mareg/}\discretionary
  {}{}{}\texttt{kbmag}}, Jan 2009.
\newblock Refereed GAP package.

\bibitem[HVL05]{havasvaughanlee}
George Havas and M.~R. Vaughan-Lee.
\newblock 4-{E}ngel groups are locally nilpotent.
\newblock {\em Internat. J. Algebra Comput.}, 15(4):649--682, 2005.

\bibitem[Mil04]{chuck}
Charles~F. Miller, III.
\newblock Combinatorial group theory.
\newblock \url{http://www.ms.unimelb.edu.au/~cfm/notes/cgt-notes.pdf}, 2004.

\bibitem[Mil13]{milne}
J.S. Milne.
\newblock {\em Group Theory}.
\newblock Available at \url{http://www.jmilne.org/math/CourseNotes/gt.html},
  version 3.13 edition, 2013.

\bibitem[MSE]{stackexchange}
User \texttt{prover}~on Mathematics Stack~Exchange.
\newblock A criterion for a group to be abelian, 3 October 2012.
\newblock \url{http://math.stackexchange.com/q/206717}

\bibitem[Neu67]{h_neumann_book}
Hanna Neumann.
\newblock {\em Varieties of groups}.
\newblock Springer-Verlag New York, Inc., New York, 1967.

\bibitem[Nic99]{nickel}
Werner Nickel.
\newblock Computation of nilpotent {E}ngel groups.
\newblock {\em J. Austral. Math. Soc. Ser. A}, 67(2):214--222, 1999.
\newblock Group theory.

\bibitem[NNN62]{neumann_cubed}
B.~H. Neumann, Hanna Neumann, and Peter~M. Neumann.
\newblock Wreath products and varieties of groups.
\newblock {\em Math. Z.}, 80:44--62, 1962.

\bibitem[Rob72]{robinson_soluble_2}
Derek J.~S. Robinson.
\newblock {\em Finiteness conditions and generalized soluble groups. {P}art 2}.
\newblock Springer-Verlag, New York-Berlin, 1972.
\newblock Ergebnisse der Mathematik und ihrer Grenzgebiete, Band 63.

\bibitem[Rob96]{robinson}
Derek J.~S. Robinson.
\newblock {\em A course in the theory of groups}, volume~80 of {\em Graduate
  Texts in Mathematics}.
\newblock Springer-Verlag, New York, second edition, 1996.

\bibitem[Tra11]{traustason}
Gunnar Traustason.
\newblock Engel groups.
\newblock In {\em Groups {S}t {A}ndrews 2009 in {B}ath. {V}olume 2}, volume 388
  of {\em London Math. Soc. Lecture Note Ser.}, pages 520--550. Cambridge Univ.
  Press, Cambridge, 2011.

\bibitem[Ven16]{gv}
Geetha Venkataraman.
\newblock Groups in which squares and cubes commute, 2016, arXiv:1605.05463.

\end{thebibliography}

\end{document}